\documentclass[11pt,reqno]{amsart}

\usepackage[utf8]{inputenc}

\usepackage{amsmath,amsthm,amssymb}
\usepackage{amssymb}

\usepackage{graphics}
\usepackage{hyperref}
\usepackage[usenames, dvipsnames]{xcolor}

\usepackage{mathtools}
\mathtoolsset{showonlyrefs}

\usepackage[square,sort,comma,numbers]{natbib}

\definecolor{darkblue}{rgb}{0.0,0.0,0.3}
\hypersetup{colorlinks,breaklinks,
  linkcolor=darkblue,urlcolor=darkblue,
anchorcolor=darkblue,citecolor=darkblue}

\theoremstyle{plain}
\newtheorem{theorem}{Theorem}
\newtheorem*{theorem*}{Theorem}
\newtheorem{lemma}[theorem]{Lemma}
\newtheorem{proposition}[theorem]{Proposition}
\newtheorem*{proposition*}{Proposition}
\newtheorem{corollary}[theorem]{Corollary}
\newtheorem*{corollary*}{Corollary}

\theoremstyle{definition}
\newtheorem{remark}[theorem]{Remark}

\renewcommand{\Im}{\operatorname{Im}}
\renewcommand{\Re}{\operatorname{Re}}
\DeclareMathOperator{\SL}{SL}
\DeclareMathOperator{\GL}{GL}

\title[Sign changes on sums of two squares]{Sign changes of cusp form
coefficients on indices that are sums of two squares}
\author{David Lowry-Duda}
\email{davidlowryduda@brown.edu}
\address{ICERM, 121 South Main Street, Box E, 11th Floor, Providence, RI 02903}
\thanks{This work was supported by the Simons Collaboration in Arithmetic
Geometry, Number Theory, and Computation via the Simons Foundation grant
546235.}

\begin{document}

\begin{abstract}
  We study sign changes in the sequence $\{ A(n) : n = c^2 + d^2 \}$, where
  $A(n)$ are the coefficients of a holomorphic cuspidal Hecke eigenform.
  After proving a variant of an axiomatization for detecting and
  quantifying sign changes introduced by Meher and Murty,
  we show that there are at least $X^{\frac{1}{4} - \epsilon}$ sign changes in
  each interval $[X, 2X]$ for $X \gg 1$. This improves to $X^{\frac{1}{2} -
  \epsilon}$ many sign changes assuming the Generalized Lindel\"{o}f Hypothesis.
\end{abstract}

\maketitle

\section{Introduction}\label{sec:intro}

Let $f = \sum A(n) n^{\frac{k-1}{2}} q^n$ denote a weight $k$ holomorphic Hecke
cuspidal eigenform on the full modular group $\SL(2, \mathbb{Z})$, normalized so that
$A(1) = 1$.
Sato-Tate equidistribution implies that the sequence $\{ A(n) \}_{\mathbb{N}}$
contains infinitely many sign changes.

Quantitatively, one can show that for $X \gg 1$, the
sequence $\{ A(n) \}_{\mathbb{N}}$ changes sign for at least one $n$ in the
interval $[X, X + X^{\frac{3}{5} + \epsilon}]$.
One way to prove this is to consider
a set of criteria first axiomatized by Meher and Murty~\cite{mehermurty14}.
For a general sequence of real numbers $\{ a(n) \}_\mathbb{N}$, Murty and Meher
specify conditions on the sizes of $a(n), \sum_{n \leq
X} a(n)$, and $\sum_{n \leq X} \lvert a(n)\rvert^2$ that allow quantitative
counts of sign changes.
The idea is simple: if $\sum_{n \leq X} a(n)$ is small while $\sum_{n \leq X}
\lvert a(n) \rvert^2$ is large (and if each $a(n)$ isn't too big), then there
must be lots of sign changes.

With the axiomatization due to Meher and Murty, sign changes in each
interval $[X, X + X^{\frac{3}{5} + \epsilon}]$ follow from Deligne's
bound~\cite{deligne71} $A(n) \ll n^{\epsilon}$, the very significant
cancellation in the partial sums~\cite{hafnerivic89, hkldw1}
$\sum_{n \leq X} A(n) \ll X^{\frac{1}{3} + \epsilon}$, and the
Rankin--Selberg 
asymptotic~\cite{rankin, selberg1}
$\sum_{n \leq X} \lvert A(n) \rvert^2 = cX + O(X^{\frac{3}{5}})$.
This result, as well as variations and other generalizations of the
Meher--Murty 
axiomatization to complex coefficients, are given in~\cite{hkldw3}.

A more difficult question is to study sign changes among sparse subsequences of
coefficients, which is related to more refined studies of the
distribution of values. Recently several such
results (for example,~\cite{kohnen2014_primepowersignchanges, gkr2015_simulsignchanges,
kohnen, hkkl2012_signchanges, jllrw2019}) have appeared.

In this paper, we investigate sign changes in
$\{ A(n) : n = c^2 + d^2, (c, d) \in \mathbb{N}^2\}$, the subsequence
consisting of coefficients whose indices are sums of two squares.
Landau~\cite{landau1908} showed that the number of integers up to $X$ that can
be written as a sum of two squares is of the order $X/\sqrt{\log X}$, and thus
this is a density $0$ subsequence. In~\cite{banerjeepandey20}, Banerjee and
Pandey show that this sequence has at least $X^{\frac{1}{8} - \epsilon}$ many
sign changes for $n \in [X, 2X]$ for $X \gg 1$.

Our first result is to strengthen this theorem.

\begin{theorem}\label{thm:main}
  Let $f = \sum A(n) n^{\frac{k-1}{2}} q^n$ be a weight $k$ holomorphic cuspidal Hecke
  eigenform on $\SL(2, \mathbb{Z})$. For any $\epsilon > 0$ and all $X \gg 1$, the
  sequence $\{ A(n) : n = c^2 + d^2, (c, d) \in \mathbb{N}^2 \}$ has has least
  one sign change among indices $n = c^2 + d^2$ with
  $n \in [X, X + X^{\frac{3}{4} + \epsilon}]$. Thus there are at least
  $X^{\frac{1}{4} - \epsilon}$ sign changes in the interval $[X, 2X]$ for all
  $X \gg 1$.
\end{theorem}

More generally, we establish criteria guaranteeing sign changes similar to the
Meher--Murty 
axiomatization (see Theorem~\ref{thm:criteria}).
We also note that we restrict to forms on $\SL(2, \mathbb{Z})$ purely for
simplicity of presentation --- the techniques apply to forms of any level.
Applying this criteria, we parametrize Theorem~\ref{thm:main} in terms of
progress towards certain moment and subconvexity bounds.
Assuming the strongest conjectured bounds, our methodology proves the
following.

\begin{theorem}\label{thm:conj}
  Assume the Generalized Lindel\"{o}f Hypothesis.
  In the notation of Theorem~\ref{thm:main}, for any $\epsilon > 0$ and all $X
  \gg 1$, the sequence $\{ A(n) : n = c^2 + d^2, (c, d) \in \mathbb{N}^2 \}$
  has has least one sign change among indices $n = c^2 + d^2$ with $n \in [X, X
  + X^{\frac{1}{2} + \epsilon}]$. There are at least $X^{\frac{1}{2} -
  \epsilon}$ sign changes in the interval $[X, 2X]$ for all $X \gg 1$.
\end{theorem}

\section{Overview of Proof}

To prove our results, we generalize~\cite{mehermurty14} and
refine the original arguments of~\cite{banerjeepandey20}.
If we wanted to use the axiomatized criteria for sign changes
from~\cite{mehermurty14}, we would study the partial sums
\begin{equation*}
  \sum_{\substack{n \leq X \\ n = c^2 + d^2}} A(n)
  \qquad \text{and} \qquad
  \sum_{\substack{n \leq X \\ n = c^2 + d^2}} A^2(n).
\end{equation*}
But it is instead much easier to study the partial sums
\begin{equation*}
  S_1(X) := \sum_{n \leq X} A(n)r_2(n)
  \qquad \text{and} \qquad
  S_2(X) := \sum_{n \leq X} A^2(n) r_2(n),
\end{equation*}
where $r_2(n) = \# \{ n = c^2 + d^2 : c, d \in \mathbb{Z} \}$ is the number of
ways of writing $n$ as a sum of $2$ squares. Multiplying by $r_2(n)$ restricts
the indices to those $n$ that can be written as sums of $2$ squares but does
not affect the sign of $A(n)$, as $r_2(n)$ is nonnegative. Thus we study sign
changes in the sequence $\{ A(c^2 + d^2) \}$ by studying sign changes of
$\{ A(n) r_2(n) \}$.

To detect these sign changes, we can adapt the axiomatization
of~\cite{mehermurty14} in terms of bounds for $A(n), S_1(X)$, and $S_2(X)$.

\begin{theorem}[Sign change criteria]\label{thm:criteria}
  Let $w(n) \geq 0$ denote a system of nonnegative weights.
  Suppose a sequence of real numbers $\{ A(n) \}_{\mathbb{N}}$ satisfies
  \begin{enumerate}
    \item $A(n) \ll n^{\alpha + \epsilon}$,
    \item $\sum_{n \leq X} A(n)w(n) \ll X^{\beta + \epsilon}$,
    \item $\sum_{n \leq X} A^2(n)w(n) = c X^{\gamma} + O(X^{\eta + \epsilon})$,
  \end{enumerate}
  for all $\epsilon > 0$, where $\alpha, \beta, c, \gamma$, and $\eta$ are
  positive real constants. Then for any $r$ with
  \begin{equation*}
    \max(\alpha + \beta, \eta) - (\gamma - 1) < r < 1,
  \end{equation*}
  the sequence $\{A(n)w(n) \}_{\mathbb{N}}$ has at least one sign change for $n$ in the interval
  $[X, X + X^r]$ for all $X \gg 1$.
\end{theorem}

\begin{proof}
  We note that all implicit constants in this proof may depend on $\epsilon,
  \alpha, \beta, \gamma$, and $\eta$. Suppose for the sake of contradiction
  that $A(n) w(n)$ has constant sign (WLOG nonnegative) for all $n \in [X, X
  + X^r]$. Choose
  \begin{equation*}
    \epsilon < \frac{r + (\gamma - 1) - \max(\alpha + \beta, \eta)}{2}.
  \end{equation*}
  Then on the one hand, we have that
  \begin{equation*}
    \sum_{X \leq n \leq X + X^r} A^2(n) w(n)
    \ll
    X^{\alpha + \epsilon} \sum_{X \leq n \leq X + X^r} A(n) w(n)
    \ll
    X^{\alpha + \beta + 2\epsilon},
  \end{equation*}
  where we note that the upper bound $r < 1$ removes further dependence on $r$.
  On the other hand, we have
  \begin{align*}
    \sum_{X \leq n \leq X + X^r} A^2(n) w(n)
    &=
    c((X + X^r)^\gamma - X^\gamma) + O(X^{\eta + \epsilon})
    \\
    &=
    \gamma c X^{\gamma - 1 + r} + O(X^{\eta + \epsilon} + X^{\gamma - 2 + 2r}).
  \end{align*}
  As $r + (\gamma - 1) > \eta + \epsilon$, these two equations imply that
  $X^{\gamma - 1 + r} \ll X^{\alpha + \beta + 2\epsilon}$, contradicting the
  hypotheses. Thus $A(n) w(n)$ changes sign at least once in the interval
  $[X, X + X^r]$.
\end{proof}

In the remainder of this paper, we apply this theorem with $w(n) = r_2(n)$.
To quantify sign changes, it is thus sufficient to study $S_1$ and $S_2$.
We can understand these sums in terms of automorphic $L$-functions. Let
$\theta(z) = \sum_{n \in \mathbb{Z}} e^{2 \pi i n^2 z}$ denote the classical
theta function, which is a weight $\frac{1}{2}$ form on $\Gamma_0(4)$.
Then $\theta^2(z) = 1 + \sum_{n \geq 1} r_2(n) e^{2 \pi i n z}$
is a modular form of weight $1$ on $\Gamma_0(4)$ with character $\chi_{-4}$,
where $\chi_{-4} = ( \frac{-1}{\cdot} )$ is the unique character mod $4$ with
$\chi_{-4}(3) = -1$.
Define the Dirichlet series
\begin{align*}
  L(s, f \otimes \theta^2) &:= L(2s, \chi_{-4}) \sum_{n \geq 1} \frac{A(n) r_2(n)}{n^s}, \\
  L(s, f \otimes f) &:= \zeta(2s) \sum_{n \geq 1} \frac{A^2(n)}{n^s}, \\
  L(s, f \otimes f \otimes \chi_{-4}) &:= L(2s, \chi_{-4}) \sum_{n \geq 1} \frac{\chi_{-4}(n) A^2(n)}{n^s}.
\end{align*}
These are understandable Rankin--Selberg 
convolution $L$-functions that converge absolutely for $\Re s > 1$ and
have meromorphic continuation to $\mathbb{C}$.

In addition, we define the Dirichlet series
\begin{equation*}
  \widetilde{L}(s, f^2 \times \theta^2)
  := \sum_{n \geq 1} \frac{A^2(n) (r_2(n) / 4)}{n^s},
\end{equation*}
which is \emph{not} an $L$-function.
The coefficients $r_2(n)/4$ are multiplicative
(as $r_2(n) / 4 = \sum_{d \mid n} \chi_{-4}(d)$). Thus the coefficients of
$\widetilde{L}(s, f^2 \times \theta^2)$ are multiplicative. Banerjee and
Pandey~\cite[Lemma~3.1]{banerjeepandey20} show that
\begin{equation}\label{eq:Ltilde_equals}
  \widetilde{L}(s, f^2 \times \theta^2)
  =
  \frac{L(s, f \otimes f)}{\zeta(2s)}
  \frac{L(s, f \otimes f \otimes \chi_{-4})}{L(2s, \chi_{-4})}
  U(s),
\end{equation}
where $U(s)$ is an Euler product that converges absolutely for $\Re s > \frac{1}{2}$.

\begin{remark}
  We note that we use the notation $L(s, f \otimes g)$ to refer to the
  Rankin--Selberg 
  $L$-function, including the associated Dirichlet $L$-function or zeta
  function. This is notationally different than in~\cite{banerjeepandey20}. We
  note that certain equations in~\cite{banerjeepandey20} have
  $\zeta(2s)$ instead of $L(2s, \chi_{-4})$ in Rankin--Selberg 
  $L$-functions, but these do not affect their results.
\end{remark}

Bounds for the parameters in Theorem~\ref{thm:criteria} are related to
subconvexity and moment bounds. Suppose $\beta$ and $\beta'$ satisfy
\begin{align}
  S_1(T) = \sum_{n \leq T} A(n) r_2(n)
  &\ll
  T^{\beta + \epsilon},%
  \label{line:beta_bound}
  \\
  \frac{1}{T} \int_{1}^T
  \lvert L(\tfrac{1}{2} + \epsilon + it, f \otimes \theta^2) \rvert dt
  &\ll
  T^{\beta' + \epsilon}%
  \label{line:betaprime_bound}
\end{align}
for all $\epsilon > 0$ as $T \to \infty$. Similarly, suppose $\eta, \eta'$,
and $\eta''$ satisfy
\begin{align}
  S_2(T) = \sum_{n \leq T} A^2(n) r_2(n)
  &=
  c_f T
  +
  O(T^{\eta + \epsilon}),%
  \label{line:eta_bound}
  \\
  \frac{1}{T} \int_{1}^T
  \lvert L(\tfrac{1}{2} + \epsilon + iT, f \otimes g) \rvert^2
  dt
  &\ll
  (1 + \lvert T \rvert)^{\eta' + \epsilon}%
  \label{line:etaprime_bound1}
  \\
  \lvert L(\tfrac{1}{2} + \epsilon + iT, f \otimes g) \rvert
  &\ll
  (1 + \lvert T \rvert)^{\eta'' + \epsilon}%
  \label{line:etaprime_bound2}
\end{align}
where the same bounds apply for both $g = f$ and $g =
f \otimes \chi_{-4}$ and for all $\epsilon > 0$ as $T \to \infty$.
The decomposition~\eqref{eq:Ltilde_equals} implies that
$\widetilde{L}(s, f^2 \otimes \theta^2)$ has a simple pole at $s = 1$ and is
otherwise analytic for $\Re s > \frac{1}{2}$, which explains the shape
of~\eqref{line:eta_bound}.
Here and below, all implicit constants are allowed to depend on $\epsilon$.

The notation $\beta$ and $\eta$ here agrees with their application via
Theorem~\ref{thm:criteria}.
We can relate bounds for $\beta', \eta'$, and $\eta''$ to bounds for $\beta$
and $\eta$.

\begin{lemma}\label{lem:prime_to_not_prime}
  In the notation just above, we have that
  \begin{equation*}
    \beta \leq 1 - \frac{1}{2(1 + \beta')}
  \end{equation*}
  and
  \begin{equation*}
    \eta \leq 1 - \frac{1}{2(1 + \max(\eta', 2\eta'' - 1))}.
  \end{equation*}
\end{lemma}

This result is straightforward, but technical.
We defer the proof to \S\ref{ssec:lem_prime_proof}.

In terms of the parameters of Theorem~\ref{thm:criteria}, Deligne's bound
shows that we can take $\alpha = \epsilon$. The pole
in~\eqref{eq:Ltilde_equals} at $s = 1$ gives that $\gamma = 1$.
It only remains to study bounds for $\beta$ and $\eta$.

Computations with the convexity bound
$\lvert L(\frac{1}{2} + \epsilon + it, f \otimes g)\rvert
\ll
(1 + \lvert t \rvert)^{1 + \epsilon}$ (which is true for each of $g = \theta^2$, $g =
f$, and $g = f \otimes \chi_{-4}$)
imply that we can take $\beta' = 1, \eta' = 2$,
and $\eta'' = 1$.
By Lemma~\ref{lem:prime_to_not_prime}, we get ``convexity bounds'' $\beta =
\frac{3}{4} + \epsilon$ and $\eta= \frac{5}{6} + \epsilon$.
As a corollary to Theorem~\ref{thm:criteria}, we get the following quantitative
sign change result.

\begin{corollary}[Sign changes from convexity bounds]
  The sequence $\{ A(n) r_2(n) \}$ has a sign change at least once in each interval $[X, X
  + X^{r}]$ for each $r > \frac{5}{6}$ and $X \gg 1$.
\end{corollary}

Better results than the convexity bound exist in the literature.
Each Rankin--Selberg 
$L$-function we use is one of Perelli's ``general $L$-functions'',
and~\cite[Theorem~4]{perelli82} shows that one can take $\eta' = 1$.

\begin{remark}
  There is a missing term in Perelli's Theorem. This is discussed
  in~\cite[\S3]{matsumoto99}. This does not affect the bound for $\eta'$.
  We note that~\cite{ivic2001meanvalues, kanemitsu2002} prove related bounds in
  different ways.
\end{remark}

Instead of bounding $\beta'$, it is better to directly produce bounds for
$S_1(X)$ directly by mimicking Rankin's original techniques to bound
coefficients of cusp forms~\cite{rankin}.
The idea of the proof is now well-known: one first applies a general method of
Landau to get good bounds on the coefficients of $L(s, f \otimes \theta^2)$,
which are essentially the desired coefficients convolved with $\chi_{-4}$.
To recover bounds for $S_1(X)$, we apply a variant of M\"{o}bius inversion and
a simple sieve.
We defer the details to \S\ref{ssec:prop_S1_proof}, but this proves the
following bound.

\begin{proposition}\label{prop:S1bound}
  With the same notation as above, we have that
  \begin{equation}
    S_1(X) = \sum_{n \leq X} A(n) r_2(n) \ll X^{\frac{3}{5} + \epsilon}.
  \end{equation}
\end{proposition}

Applying Lemma~\ref{lem:prime_to_not_prime} and Theorem~\ref{thm:criteria} with
$\eta' = 1$ (the Perelli-type bound), $\eta'' = 1$ (the convexity bound), and
$\beta = \frac{3}{5}$ (the bound from Proposition~\ref{prop:S1bound}) gives the
following.

\begin{corollary}\label{cor:simple_bound}
  The sequence $\{ A(n) r_2(n) \}$ change signs at least once in each interval $[X, X
  + X^{r}]$ for each $r > \frac{3}{4}$ and $X \gg 1$.
\end{corollary}

This implies Theorem~\ref{thm:main} from \S\ref{sec:intro}.
We note that the improved bound from Proposition~\ref{prop:S1bound} does not yield an
improved sign change result, but instead reveals that the current barrier
towards better results is understanding $\eta$.

\section{Conjectural Limits of these Techniques}\label{sec:conj}

There has been substantial work on subconvexity bounds for $\GL(2)$ type
$L$-functions and $\GL(2) \times \GL(2)$ type Rankin--Selberg 
convolutions. In extreme generality, Michel and
Venkatesh~\cite{michel_venkatesh_gl2subconvexity} have shown that there is a
computable exponent $\delta$ such that
$L(\frac{1}{2} + it, f \otimes g) \ll (1 + \lvert t \rvert)^{1 - \delta}$,
which implies that there is some choice $\eta'' < 1$ for~\eqref{line:etaprime_bound2}.

More recently, improved explicit bounds have begun to appear. In his dissertation,
Raju~\cite[Theorem~3.1.1]{raju2019circle} proves that
\begin{equation*}
  L(\tfrac{1}{2} + it, f \times g)
  \ll_\epsilon
  (1 + \lvert t \rvert)^{1 - \frac{1}{34} + \epsilon}
\end{equation*}
using variants of the circle method. In their
preprint~\cite[Theorem~1]{acharya2020taspect}, Acharya, Sharma, and Singh use similar
methods to prove
\begin{equation*}
  L(\tfrac{1}{2} + it, f \times g)
  \ll_\epsilon
  (1 + \lvert t \rvert)^{1 - \frac{1}{16} + \epsilon}.
\end{equation*}
This suggest that we can take $\eta''$ to be $1 - \frac{1}{16}$. Unfortunately, this
does not yield an improved bound for $\eta$ through
Lemma~\ref{lem:prime_to_not_prime}, but instead highlights the averaged second moment
bound~\eqref{line:etaprime_bound1} as the primary obstruction towards improvements to
Theorem~\ref{thm:main}.

If we assume the Generalized Lindel\"of Hypothesis, we would have
$\beta' = \eta' = \eta'' = 0$. Lemma~\ref{lem:prime_to_not_prime} then implies that
both $\beta$ and $\eta$ are bounded above by $\frac{1}{2}$.
The parametrization in Theorem~\ref{thm:criteria} then gives the following,
which implies Theorem~\ref{thm:conj} from \S\ref{sec:intro}.

\begin{theorem}
  Let $f = \sum A(n) n^{\frac{k-1}{2}} q^n$ be a weight $k$ holomorphic
  cuspidal Hecke eigenform on $\Gamma_0(N)$.
  Assume the Generalized Lindel\"of Hypothesis.
  For any $\epsilon > 0$ and all $X \gg 1$, we have that
  \begin{align}
    \sum_{n \leq X} A(n) r_2(n) &\ll_\epsilon X^{\frac{1}{2} + \epsilon},%
    \label{eq:l1_conj}
    \\
    \sum_{n \leq X} A^2(n) r_2(n) &= c_f X + O_\epsilon(X^{\frac{1}{2} + \epsilon}),%
    \label{eq:l2_conj}
  \end{align}
  for a constant $c_f$ depending on $f$.
  Further, the sequence $\{ A(n) : n = c^2 + d^2, (c, d) \in \mathbb{N}^2 \}$
  has at least one sign change among the indices $n \in [X, X +
  X^{\frac{1}{2} + \epsilon}]$, and there are at least $X^{\frac{1}{2} -
  \epsilon}$ sign changes in the interval $[X, 2X]$.
\end{theorem}

Thus the limit of this methodology is square-root type cancellation
in~\eqref{eq:l1_conj} and in the error term of~\eqref{eq:l2_conj}, and
approximately $\sqrt{X}$ many sign changes in intervals $[X, 2X]$.

\begin{remark}
  It is not clear whether we should expect~\eqref{eq:l1_conj}
  and~\eqref{eq:l2_conj} to be \emph{tight} bounds (up to $\epsilon$ powers).
  As the signs of $A(n)$ vary randomly, we would heuristically expect an error
  term \emph{at most} $\frac{1}{2} + \epsilon$ in~\eqref{eq:l1_conj}, but it's
  not clear that this should remain true when restricted to the subsequence of
  indices representable as sums of two squares.
  The Dirichlet series $L(s, f \otimes \theta^2)/L(2s,
  \chi_{-4})$, which we used to study~\eqref{eq:l1_conj}, has clear meromorphic
  continuation to $\mathbb{C}$.
  We expect that the critical zeros of $L(2s, \chi_{-4})$ in the denominator
  lead to poles, in which case~\cite[Theorem~1]{dld2020irregularity} implies that
  \begin{equation*}
    \sum_{n \leq X} A(n) r_2(n) = \Omega_{\pm}(X^{\frac{1}{4}}),
  \end{equation*}
  i.e.\ that the sum is of magnitude at least $X^{1/4}$ both positively and
  negatively infinitely often. There is a lot of space between $\frac{1}{4}$
  and $\frac{1}{2}$.

  In the Dirichlet series $\widetilde{L}$ (used to study~\eqref{eq:l2_conj}), the
  correction Euler product $U(s)$ only converges for $\Re s > \frac{1}{2}$, and is
  otherwise difficult to understand. There is little hope in gaining better
  understanding of this error term using Dirichlet series techniques.
\end{remark}

\section{Technical Proofs}

\subsection{Proof of Proposition~\ref{prop:S1bound}}\label{ssec:prop_S1_proof}

To prove Proposition~\ref{prop:S1bound}, we first construct $L(s, f \otimes
\theta^2)$ as a Rankin--Selberg 
convolution. Let $E_k(z, w)$ denote
\begin{equation}
  E_k(z, w) := \sum_{\gamma \in \Gamma_\infty \backslash \Gamma_0(4)}
  \chi_{-4}(\gamma) J(\gamma, z)^{-k} \Im (\gamma z)^w,
\end{equation}
the weight $k$ real analytic Eisenstein series on $\Gamma_0(4)$. For a matrix
$\big( \begin{smallmatrix} a & b \\ c & d \end{smallmatrix} \big) \in \Gamma_0(4)$,
$J(\gamma, z)$ denotes the normalized transformation factor
$J(\gamma, z) = (cz + d) / \lvert cz + d \rvert$ and
$\chi_{-4}(\gamma) = \chi_{-4}(d)$. This Eisenstein series has completion
$E_k^*(z, w) = \Lambda(2w, \chi_{-4}) E_k(z, w)$, where
$\Lambda(w, \chi_{-4}) = (4/\pi)^{\frac{w}{2}} \Gamma(\frac{w+1}{2}) L(w, \chi_{-4})$
is the completed Dirichlet $L$-function. The completed Eisenstein series $E^*(z, w)$
is entire and satisfies the functional equation $E_k^*(z, 1-w) = E_k^*(z, w)$.
(see~\cite{osullivan_formulasforeisenstein},~\cite[\S8]{stromberg_modularformsclassicalapproach},
or the author's thesis~\cite[\S2]{LowryDudaThesis} for similar
computations with weighted Eisenstein series).

Let $V(z)$ denote the product $V(z) = y^{\frac{k+1}{2}} f(z) \overline{\theta^2(z)}$.
Recall that $\theta(z)$ transforms with a prototypical half-integral weight
multiplier~\cite{shimura}, $\theta(\gamma z) = \widetilde{j}(\gamma, z) \theta(z)$, where
$\widetilde{j}(\gamma, z) = \varepsilon_d^{-1} \big( \frac{c}{d} \big) (cz +
d)^{\frac{1}{2}}$, where $\varepsilon_d = 1$ for $d \equiv 1 \bmod 4$ and
$\varepsilon_d = i$ for $d \equiv 3 \bmod 4$.
Thus $\theta^2(\gamma z) = \chi_{-4}(\gamma) (cz + d)$.

A short computation shows that $V(z)$ and $E_{k-1}(z, w)$ transform like weight
$(k-1)$ modular forms with character $\chi_{-4}$ with respect to the normalized
factor of automorphy $J(\gamma, z)$. Consider the Petersson inner product $\langle V(z),
E_{k-1}(z, w) \rangle$, which converges since $f$ is a holomorphic cuspform.

The classical ``unfolding'' argument constructs the Rankin--Selberg 
$L$-function:
\begin{align*}
  \langle &V(z), E_{k-1}(z, w) \rangle
  =
  \iint_{\Gamma_\infty \backslash \Gamma_0(4)}
  V(z) \overline{E_{k-1}(z, w)} d \mu(z)
  \\
  &= \int_0^\infty \int_0^1
  f(z) \overline{\theta^2}(z) y^{w + \frac{k+1}{2}} \frac{dx \; dy}{y^2}
  \\
  &= \int_0^\infty
  \sum_{\substack{m \geq 1 \\ n \geq 0}} A(n) n^{\frac{k-1}{2}} r_2(m)
  \Big(
    \int_0^1 e^{2 \pi i (n - m) x} dx
  \Big)
  e^{-2 \pi (n + m) y} y^{w + \frac{k-1}{2}} \frac{dy}{y}
  \\
  &=
  \frac{\Gamma(w + \frac{k-1}{2})}
       {(4 \pi)^{w + \frac{k-1}{2}}}
  \sum_{n \geq 1} \frac{A(n) r_2(n)}{n^w}.
\end{align*}
Thus we define the completed Rankin--Selberg 
$L$-function
\begin{align*}
  \Lambda(s, f \otimes \theta^2)
  :=&
  (4 \pi)^{\frac{k-1}{2}} \Lambda(2s, \chi_{-4})
  \frac{\Gamma(s + \frac{k-1}{2})}
       {(4 \pi)^{s + \frac{k-1}{2}}}
  \sum_{n \geq 1} \frac{A(n) r_2(n)}{n^s}
  \\
  =&
  \frac{1}{\pi^{2s}}
  \Gamma(s + \tfrac{k-1}{2}) \Gamma(s + \tfrac{1}{2})
  L(2s, \chi_{-4})
  \sum_{n \geq 1} \frac{A(n) r_2(n)}{n^s},
\end{align*}
which inherits the functional equation
$\Lambda(s, f \otimes \theta^2) = \Lambda(1-s, f \otimes \theta^2)$
from $E^*_{k-1}(s, z)$.

Define the coefficients
\begin{equation}
  L(s, f \otimes \theta^2)
  = L(2s, \chi_{-4})  \sum_{n \geq 1} \frac{A(n) r_2(n)}{n^s}
  =: \sum_{n \geq 1} \frac{b(n)}{n^s}.
\end{equation}
We apply Landau's method (as applied to general Dirichlet series with functional
equation in~\cite{chand}) to get bounds on partial sums of $b(n)$.
In the notation of~\cite{chand}, we have $\delta = 1, A = 2$, and $\Delta(s) =
\Gamma(s + \frac{k-1}{2})\Gamma(s + \frac{1}{2})$, which implies that
\begin{equation}\label{eq:from_CN}
  \sum_{n \leq X} b(n)
  \ll
  \sum_{X < n < X + O(y)} \lvert b(n) \rvert
  +
  X^{\frac{3}{2} + \epsilon} y^{-\frac{3}{2}}.
\end{equation}
As
\begin{equation}
  b(n) = \sum_{d^2 \mid n} \chi_{-4}(d^2) A(n / d^2) r_2(n / d^2),
\end{equation}
we have that $b(n) \ll n^{\epsilon}$. Thus~\eqref{eq:from_CN} can be written as
\begin{equation*}
  \sum_{n \leq X} b(n)
  \ll
  y^{1 + \epsilon}
  +
  X^{\frac{3}{2} + \epsilon} y^{-\frac{3}{2}}.
\end{equation*}
This is balanced by choosing $y = X^{3/5}$, giving that
\begin{equation}\label{eq:bn_bound}
  \sum_{n \leq X} b(n) \ll X^{\frac{3}{5} + \epsilon}.
\end{equation}

\begin{remark}
  The $\epsilon$ power can be removed, giving an error term $O(X^{3/5})$, by
  applying a more technical analysis in Landau's method.
  This follows from the main theorem of the
  preprint~\cite{lowrydudathornetani}, which uniformizes~\cite{chand}.
  This would not improve the quantitative sign change result.
\end{remark}

To complete the proof of Proposition~\ref{prop:S1bound}, we recover bounds on
$S_1(X)$ from~\eqref{eq:bn_bound}. We do this with a form of M\"{o}bius inversion.

\begin{lemma}
  Suppose $a(n)$ and $b(n)$ are arithmetic functions satisfying the relation
  \begin{equation*}
    b(n) = \sum_{d^2 \mid n} a(n/d^2) \chi(d^2)
  \end{equation*}
  for a Dirichlet character $\chi$. Then
  \begin{equation*}
    a(n) = \sum_{d^2 \mid n} \mu(d) \chi(d^2) b(n/d^2).
  \end{equation*}
\end{lemma}

This is a twisted form of the inversion implicitly used in~\cite[\S4]{rankin}.

\begin{proof}
  We compute directly that
  \begin{align*}
    \sum_{d^2 \mid n} \mu(d) \chi(d^2) b(n/d^2)
    &=
    \sum_{d^2 \mid n} \mu(d) \chi(d^2) \sum_{\ell^2 \mid \frac{n}{d^2}} a(n/d^2 \ell^2) \chi(\ell^2)
    \\
    &=
    \sum_{n = d^2 \ell^2 k} \mu(d) \chi^2(d\ell) a(k)
    =
    \sum_{n = D^2 k} \chi^2(D) a(k) \sum_{d \mid D} \mu(d)
    \\
    &=
    \chi^2(1) a(n) = a(n).
  \end{align*}
\end{proof}

It follows that
\begin{equation*}
  A(n) r_2(n) = \sum_{d^2 \mid n} \mu(d) \chi_{-4}(d^2) b(n / d^2),
\end{equation*}
and thus
\begin{align*}
  S_1(X) = \sum_{n \leq X} A(n) r_2(n)
  &=
  \sum_{n \leq X} \sum_{d^2 \mid n} \mu(d) \chi_{-4}^2(d) b(n / d^2)
  \\
  &=
  \sum_{d \leq \sqrt{X}} \mu(d) \chi_{-4}^2(d) \sum_{m \leq \frac{X}{d^2}} b(m)
  \\
  &\ll_\epsilon
  \sum_{d \leq \sqrt{X}} \mu(d) \chi_{-4}^2(d)
  \Big( \frac{X^{\frac{3}{5} + \epsilon}}{d^{\frac{6}{5} + \epsilon}} \Big)
  \ll_\epsilon X^{\frac{3}{5} + \epsilon}
\end{align*}
as the $d$ sum is absolutely bounded by $\zeta(6/5)$. This completes the proof of
Proposition~\ref{prop:S1bound}.

\subsection{Proof of Lemma~\ref{lem:prime_to_not_prime}}\label{ssec:lem_prime_proof}

We apply an explicit truncated Perron's formula, similar to the analysis
in~\cite{banerjeepandey20}. We use the presentation of Perron's formula
in Theorem~5.1, Theorem~5.2, and Corollary~5.3
of~\cite{montgomery_multnumtheory}. These imply that
\begin{align*}
  \sum_{n \leq X}{}^{'} A(n) r_2(n)
  &=
  \frac{1}{2\pi i} \int_{\sigma - iT}^{\sigma + iT}
  \frac{L(s, f \otimes \theta^2)}{L(2s, \chi_{-4})} \frac{X^s}{s} ds
  +
  O\Big(\frac{X^{1 + \epsilon}}{T}\Big),
  \\
  \sum_{n \leq X}{}^{'} A^2(n) r_2(n)
  &=
  \frac{1}{2 \pi i} \int_{\sigma - iT}^{\sigma + iT}
  \widetilde{L}(s, f^2 \times \theta^2) \frac{X^s}{s} ds
  +
  O\Big(\frac{X^{1 + \epsilon}}{T}\Big),
\end{align*}
for $\sigma = 1 + \epsilon$ for any $\epsilon > 0$, where the primes on the
sums indicate that the last term is to be divided by $2$ if $X$ is an
integer, and where we have used that $A(n) r_2(n) \ll n^{\epsilon}$ in the
remainder term approximations.
In both cases, we move the line of integration to $\frac{1}{2} + \epsilon$ and apply
Cauchy's Residue Theorem to get that
\begin{align*}
  \sum_{n \leq X}{}^{'} A(n) r_2(n)
  &=
  \Big(\int_{\mathrm{left}} + \int_{\mathrm{top}} + \int_{\mathrm{bot}}\Big)
  \frac{L(s, f \otimes \theta^2)}{L(2s, \chi_{-4})} \frac{X^s}{s} ds
  +
  O\Big(\frac{X^{1 + \epsilon}}{T}\Big),
  \\
  \sum_{n \leq X}{}^{'} A^2(n) r_2(n)
  &=
  \Big(\int_{\mathrm{left}} + \int_{\mathrm{top}} + \int_{\mathrm{bot}}\Big)
  \widetilde{L}(s, f^2 \times \theta^2)
  \frac{X^s}{s} ds
  +
  c_f X
  +
  O\Big(\frac{X^{1 + \epsilon}}{T}\Big),
\end{align*}
where the integrals are over the line segments
$[1 + \epsilon + iT, \frac{1}{2} + \epsilon + iT]$,
$[\frac{1}{2} + \epsilon - iT, 1 + \epsilon - iT]$, and
$[\frac{1}{2} + \epsilon + iT, \frac{1}{2} + \epsilon - iT]$
(respectively the left, top, and bottom sides of the box).
In this expression, $c_f$ is the residue of $\widetilde{L}$ at the simple pole at
$s = 1$.

We first consider $S_1(X)$. With the classical bound
$1 / \lvert L(1 + \epsilon + i t, \chi_{-4})\rvert
\ll \log (2 + \lvert t \rvert)$ as $\lvert t \rvert \to \infty$, we bound the
left integral via
\begin{align*}
  \int_{\mathrm{left}} \frac{L(s, f \otimes \theta^2)}{L(2s, \chi_{-4})} \frac{X^s}{s} ds
  &\ll
  X^{\frac{1}{2} + \epsilon}
  \Big(
    1 + \int_1^T \frac{%
      \log(t) \lvert L(\frac{1}{2} + \epsilon + it, f \otimes \theta^2) \rvert
    }{t} dt
  \Big)
  \\
  &\ll
  X^{\frac{1}{2} + \epsilon}
  \Big(
    1 +
    \sum_{n \leq \log_2 T}
    \frac{2^{n+1}}{T^{1 - \epsilon}} \int_{T/2^{n+1}}^{T/2^n}
    \lvert L(\tfrac{1}{2} + \epsilon + it, f \otimes \theta^2) \rvert dt
  \Big)
  \\
  &\ll
  X^{\frac{1}{2} + \epsilon}
  +
  X^{\frac{1}{2} + \epsilon}
  T^{\beta' + \epsilon},
\end{align*}
where in the last steps we bounded the integral dyadically
using~\eqref{line:betaprime_bound}.

For the top and bottom segments, we use the na\"{\i}ve
convexity bound. The explicit functional equation is given in
\S\ref{ssec:prop_S1_proof}, but we use little more than the fact that
$L(s, f \otimes \theta^2)$ is of degree $4$. The convexity bound is
\begin{equation}
  L(\sigma + it, f \otimes \theta^2)
  \ll
  (1 + \lvert t \rvert)^{2 (1 - \sigma) + \epsilon}
\end{equation}
for $0 < \sigma < 1$. Thus the top and bottom segments satisfy
\begin{align*}
  \Big(\int_{\mathrm{top}} &+ \int_{\mathrm{bot}}\Big)
  \frac{L(s, f \otimes \theta^2)}{L(2s, \chi_{-4})} \frac{X^s}{s} ds
  \\
  &\ll
  \int_{\frac{1}{2} + \epsilon}^{1 + \epsilon}
  \lvert L(\sigma + iT, f \otimes \theta^2) \rvert \log(T)
  \frac{X^{\sigma}}{T} d\sigma
  \\
  &\ll
  \int_{\frac{1}{2} + \epsilon}^{1 + \epsilon} \log T
  \frac{\max_{\frac{1}{2} + \epsilon < \sigma < 1 + \epsilon}
        \lvert L(\sigma + iT, f \otimes \theta^2) X^\sigma \rvert}
       {T}
  d\sigma
  \\
  &\ll
  X^{\frac{1}{2} + \epsilon}
  +
  \frac{X^{1 + \epsilon}}{T},
\end{align*}
where we have used the bound $\log T \ll T^{\epsilon}$.

In total, the sum is bounded by
\begin{equation*}
  S_1(X)
  \ll
  \sum_{n \leq X}{}^{'} A(n) r_2(n)
  \ll
  X^{\frac{1}{2} + \epsilon}
  T^{\beta' + \epsilon}
  +
  \frac{X^{1 + \epsilon}}{T}
  \ll
  X^{1 - \frac{1}{2(\beta' + 1)} + \epsilon},
\end{equation*}
where we chose $T = X^{1/2(\beta' + 1)}$ to balance the error terms.

We treat $S_2(X)$ more carefully as $\widetilde{L}$ is not an $L$-function and has
much worse trivial bounds.
Recall the decomposition for $\widetilde{L}$ from~\eqref{eq:Ltilde_equals}. For $\Re
s > \frac{1}{2} + \epsilon$, we note that
\begin{equation*}
  \frac{U(s)}{\zeta(2s) L(2s, \chi_{-4})} \ll (1 + \lvert t \rvert)^\epsilon
\end{equation*}
follows from the absolute convergence for $U(s)$ and classical bounds for Dirichlet
$L$-functions and the zeta function. The Cauchy-Schwarz inequality gives
\begin{align*}
  \int_{\mathrm{left}} &\widetilde{L}(s, f^2 \times \theta^2) \frac{X^s}{s} ds
  \\
  &\ll
  X^{\frac{1}{2} + \epsilon}
  \Big(
    1
    +
    \int_{1}^T
    \frac{%
      \lvert
        L(\tfrac{1}{2} + \epsilon + it, f \otimes f)
      \lvert^2
    }{t}
    t^\epsilon
    dt
  \Big)^{\frac{1}{2}}
  \\
  &\qquad \times
  \Big(
    1
    +
    \int_{1}^T
    \frac{%
      \lvert
        L(\tfrac{1}{2} + \epsilon + it, f \otimes f \otimes \chi_{-4})
      \rvert^2
    }{t}
    t^\epsilon
    dt
  \Big)^{\frac{1}{2}}.
\end{align*}
Analogous to before, we dyadically split the integrals and bound each integral
with~\eqref{line:etaprime_bound1}. Collecting these
bounds gives that
\begin{equation*}
  \int_{\mathrm{left}} \widetilde{L}(s, f^2 \times \theta^2) \frac{X^s}{s} ds
  \ll
  X^{\frac{1}{2} + \epsilon}
  +
  X^{\frac{1}{2} + \epsilon} T^{\eta' + \epsilon}.
\end{equation*}

For the top and bottom segments, it insufficient to use the na\"{\i}ve
convexity bound. Instead, we use the convexity bound that follows
from the subconvexity estimate~\eqref{line:etaprime_bound2},
\begin{equation*}
  \widetilde{L}(\sigma + it, f^2 \times \theta^2)
  \ll
  (1 + \lvert t \rvert)^{4\eta''(1 - \sigma) + \epsilon}
\end{equation*}
for $\frac{1}{2} < \sigma < 1$. Taking absolute values and bounding shows that
\begin{align*}
  \Big(\int_{\mathrm{top}} &+ \int_{\mathrm{bot}}\Big)
  \widetilde{L}(s, f^2 \times \theta^2) \frac{X^s}{s} ds
  \\
  &\ll
  \int_{\frac{1}{2} + \epsilon}^{1 + \epsilon} T^\epsilon
  \frac{\max_{\frac{1}{2} + \epsilon < \sigma < 1 + \epsilon}
        \lvert \widetilde{L}(\sigma + iT, f^2 \times \theta^2) X^\sigma \rvert}
       {T}
  d\sigma
  \\
  &\ll
  X^{\frac{1}{2} + \epsilon} T^{2\eta'' - 1 + \epsilon}
  +
  \frac{X^{1 + \epsilon}}{T^{1 - \epsilon}}.
\end{align*}

In total, we find that
\begin{align*}
  S_2(X)
  &\ll
  \sum_{n \leq X}{}^{'} A^2(n) r_2(n)
  \ll
  c_f X
  +
  X^{\frac{1}{2} + \epsilon} \big(
    T^{\eta' + \epsilon} + T^{2\eta'' - 1 + \epsilon}
  \big)
  +
  \frac{X^{1 + \epsilon}}{T^{1 - \epsilon}}
  \\
  &\ll
  c_f X
  +
  X^{1 - \frac{1}{2(1 + \max(\eta', 2\eta'' - 1))} + \epsilon}
\end{align*}
after optimizing the bound by choosing
$T = X^{1/2(1 + \max(\eta', 2\eta'' - 1))}$.

\vspace{5 mm}
{\small
\bibliographystyle{alpha}
\bibliography{bibfile}
}

\end{document}